\documentclass[preprint,12pt]{elsarticle}

\usepackage{amssymb}
\usepackage{amsmath}
\usepackage{amsfonts}
\usepackage{graphicx}
\usepackage{graphics}
\usepackage{tikz}
\usepackage{titlesec}

\newtheorem{theorem}{Theorem}

\newtheorem{corollary}[theorem]{Corollary}

\newtheorem{example}[theorem]{Example}

\newtheorem{lemma}[theorem]{Lemma}

\newtheorem{proposition}[theorem]{Proposition}
\newtheorem{remark}[theorem]{Remark}

\def\qed{\vbox{\hrule
 \hbox{\vrule\hbox to 5pt{\vbox to 8pt{\vfil}\hfil}\vrule}\hrule}}

\usepackage{tikz}
\usepackage{tkz-graph}

\journal{xxxxxx}

\begin{document}
\begin{frontmatter}

\title{Bounds for Different Spreads of Line and Total Graphs }

\author{Enide Andrade}
\ead{enide@ua.pt}
\address{CIDMA-Center for Research and Development in Mathematics and Applications
         Departamento de Matem\'atica, Universidade de Aveiro, 3810-193, Aveiro, Portugal.}

\author {Eber Lenes}
\address{Grupo de Investigaci\'on Deartica, Universidad del Sin\'u.
El\'ias Bechara Zain\'um, Cartagena, Colombia.}
\ead{elenes@unisinucartagena.edu.co}

\author{Exequiel Mallea-Zepeda}
\address{ Departamento de Matem\'atica, Universidad de Tarapac\'a,
 Arica, Chile.}
\ead{ emallea@uta.cl}

\author[]{Mar\'{\i}a Robbiano \corref{cor1}}
\cortext[cor1]{Corresponding author}
\address{Departamento de Matem\'{a}ticas, Universidad Cat\'{o}lica del Norte, Av. Angamos 0610 Antofagasta, Chile.}
\ead{mrobbiano@ucn.cl}

\author{Jonnathan Rodr\'{\i}guez Z.}
\ead{jrodriguez01@ucn.cl}
\address{Departamento de Matem\'{a}ticas, Universidad Cat\'{o}lica del Norte, Av. Angamos 0610 Antofagasta, Chile.}

\begin{abstract}
In this paper we explore some results concerning the spread of the line and the total graph of a given graph. 
In particular, it is proved that for an $(n,m)$ connected graph $G$  with $m > n \geq 4$ the spread of $G$ is less than or equal to 
the spread of its line graph, where the equality holds if and only if $G$ is  regular bipartite. A sufficient condition for the spread of the graph not be greater than the signless Laplacian spread for a class of non bipartite and non regular graphs is proved.
Additionally, we derive an upper bound for the spread of the line graph of graphs 
on $n$ vertices having a vertex (edge) connectivity less than or equal to a positive integer $k$.
Combining techniques of interlacing of eigenvalues, we derive lower bounds for the Laplacian and signless Laplacian spread of the total graph of a connected graph. Moreover, for a regular graph, an upper and lower bound for the spread of its total graph is given.
\end{abstract}

\begin{keyword}
Matrix Spread; Graph Spread; $Q$-spread; Total Graph; Line Graph; Connectivity
\MSC 05C50 \sep 15A18
\end{keyword}

\end{frontmatter}

\section{Introduction}

Let $G$ be a simple undirected graph with vertex set $V\left( G\right)$ of
cardinality $n$ and edge set $E\left( G\right)$ of cardinality $m$. We say that $G$ is an $(n,m)$ graph. After the labeling of vertices, a vertex is named by its label and an edge with end vertices $i$ and $j$ is denoted by $ij$.
The number of adjacent vertices to $i$ is
called the degree of $i$ and is denoted by $d_G(i)$ (or simply $d(i)$). A graph $G$ is called $r$-regular if each vertex has degree $r$. For a finite set $U$, $|U|$ denotes its cardinality.
If $U  \subseteq V(G),$ $G-U$ denotes the subgraph of $G$ induced by $V(G) \backslash U.$ The diameter of a connected graph $G$ is the maximum distance between two vertices of $G$ and it will be denoted by $diam(T)$.
For two disjoint graphs $G_{1}$ and $G_{2}$, the join of $G_{1}$ and $%
G_{2}$ is the graph $G_{1}\vee G_{2}$ obtained from their union including all edges between the vertices in $G_1$ and the vertices in $G_{2}$.
The adjacency matrix, Laplacian matrix and signless Laplacian matrix associated with a graph $G$ is denoted by $A(G)$, $L(G)=D(G)-A(G)$ and $Q(G)=D(G)+A(G)$, respectively, where $D(G)$ is the diagonal matrix of vertex degrees.
The eigenvalues of a graph $G$ are the eigenvalues of its adjacency matrix, denoted by $\lambda_i= \lambda_i(G)$ and ordered in nonincreasing way as $ \lambda_{1}\geq \ldots \geq \lambda_{n}.$ Moreover, the eigenvalues of $L(G)$ and $Q(G)$ are also ordered in nonincreasing way as follows $\mu_1 \geq \ldots \geq \mu_{n-1} \geq \mu_{n}=0,$ and $q_1 \geq \ldots \geq q_{n}, $ respectively. These are called \textit{Laplacian eigenvalues} and \textit{signless Laplacian eigenvalues} of $G$, respectively.
Recall that $0$ is always a Laplacian
eigenvalue and its multiplicity is the number of connected components of $G$. The multiplicity of $0$ as a signless Laplacian eigenvalue of a graph $G$ without isolated vertices corresponds to the number of bipartite components of $G,$ see e.g. \cite{HaemerBk}.
The spectra of $Q\left( G\right) $ and $L\left( G\right) $ coincide if and
only if $G$ is a bipartite graph, see e.g. \cite{Domingos,lapl1,lapl2}.
In this work, $K_{n}$ and $C_{n}$ denote the complete graph and the cycle of order $n$, respectively. Moreover, $I_m$ denotes the identity matrix of order $m$ and, for a matrix $A$, $A^{t}$ denotes its transpose.
The line graph $\mathcal{L}(G)$ is the graph whose vertex set are the edges
in $G$, where two vertices are adjacent if the corresponding edges in $G$
have a common vertex.
The total graph $\mathcal{T}(G)$ is the graph whose vertex set corresponds
to the vertices of $G\cup\mathcal{L}(G)$ and two vertices are adjacent in $\mathcal{T}(G)$ if their corresponding elements are either adjacent in $G \cup \mathcal{L}(G)$ or incident in $G$. 
In \cite{BelizadRadjavi} the author presented a characterization of the structure of regular total graphs as well as other properties. In \cite{Belizad} it were considered non regular graphs and presented a method that enables to determine whether a graph is total or not.
Moreover, a relationship between the spectra of a regular graph and its total graph was presented by Cvetkovi\'c in \cite{cvetc}.
Many graph theoretical parameters have been used to describe the stability of communication networks. Tenacity is one of these parameters, which shows not only the difficulty to break down the network but also the damage that has been caused. Total graphs are the largest graphs formed by the adjacent relations of elements of a given graph. Thus, total graphs are highly recommended for the design of interconnection networks. For instance, in \cite{tenacity} the authors determine the tenacity of the total graph of a path, cycle and complete bipartite graph, and thus give a lower bound of the tenacity for the total graph of a graph. It is also worth to recall that total graphs are generalizations of line graphs.

The paper is organized as follows. In Section 2 the definitions of spread, Laplacian spread, signless Laplacian spread and the spread of the line graph are presented. An upper bound for the spread of the line graph is obtained. In Section 3, relations among the spread of the line graph and the signless Laplacian spread are given. Additionally, relations between the previous spreads and the spread of a graph are studied. Furthermore, a sufficient condition for the spread of a graph to be not greater the spread of the signless Laplacian spread for a class of non bipartite and non regular graphs is proved. Moreover, it is derived an upper bound for the spread of the line graph of
graphs with $n$ vertices having a vertex (edge) connectivity less than or equal to $k$.
This bound is attained if and only if $G \cong K_{1} \vee (K_{k} \cup K_{n-k-1}), $ where $K_{k}$ is the complete graph of order $k.$
In Section 4, using the Laplacian and signless Laplacian matrices
of the total graph of a connected graph and applying the
interlacing of eigenvalues, due to Haemers \cite{haemers2}, we obtain lower bounds for the spread, signless Laplacian spread and Laplacian spread of total graphs. Moreover, in the case of a regular graph $G$ we present an upper and lower bound for the spread of the total graph associated to $G.$

\section{Preliminaries}

In this section, we list some of the definitions of different spreads and previously known results that will be needed throughout the paper.

%
It is known that the matrices $Q(G)$ and $2I_m+A({\mathcal{L}(G)})$ share the same nonzero eigenvalues \cite{Harary}.
As a consequence, we have the following result.

\begin{lemma}
\label{381}\rm{\cite{HaemerBk} } Let $G$ be an $(n, m)$ graph with $m\geq1$
edges. Let $q_i$ be the $i$-$th$ greatest signless Laplacian eigenvalue of $G$
and $\lambda_i(\mathcal{L}(G))$ the $i$-$th$ greatest eigenvalue of its line
graph $\mathcal{L}(G)$. Then
\begin{equation*}
q_i=\lambda_i(\mathcal{L}(G))+2,
\end{equation*}
for $i=1,2,\ldots,k,$ where $k=\min\{n,m\}$. In addition, if $m>n$, then $%
\lambda_i(\mathcal{L}(G))=-2$ for $i\geq n+1$ and if $n>m$, then $q_i=0$ for
$i\geq m+1$.
\end{lemma}
The Zagreb index of $G$, $Z_g(G),$ (see \cite{GutmanDas}),
is defined as
\begin{equation*}
Z_g(G)=\sum_{i\in V(G)}d^2(i).
\end{equation*}

\begin{lemma}
\label{zagreb} \rm{ \cite{Harary}} If $G$ is a graph with $m$ edges, then the number of edges of $\mathcal{L}(G)$ is given by

\begin{eqnarray}
\label{nÂ°edges}
\theta:=\frac{Z_{g}\left( G\right)}{2}-m.
\end{eqnarray}
\end{lemma}

The \emph{spread} of an $n\times n$ complex Hermitian matrix $M$ with eigenvalues $%
\lambda _{1},\ldots ,\lambda _{n}$ is
\begin{equation*}
S\left( M\right) =\max_{i,j}\left\vert \lambda _{i}-\lambda _{j}\right\vert ,
\end{equation*}%
where the maximum is taken over all pairs of distinct eigenvalues of $M$.

The \textit{spread of the graph} $G,$ \cite{G}, with eigenvalues
$ \lambda_{1}\geq \ldots \geq \lambda_{n}$ is defined as the spread of its adjacency matrix, that is:
\begin{equation*}
S\left( G\right) =\lambda _{1}-\lambda _{n}.
\end{equation*}

Let $K_{a,b}$ be the complete bipartite graph where the bipartition of its vertex set has $a$ vertices in one subset and $b$ vertices in the other. The next result can be seen in \cite{G}.

\begin{lemma}
\label{model} \rm{\cite{G}} Let $G$ be an $(n,m)$ graph. Then
\begin{equation*}
S\left( G\right) \leq \lambda _{1}+\sqrt{2m-\lambda _{1}^{2}}\leq 2\sqrt{m}.%
\end{equation*}
Equality holds throughout if and only if equality holds in the first
inequality; equivalently, if and only if $m=0$ or $G=K_{a,b},$ for some $%
a,b$ with $m=ab$ and $a+b\leq n.$
\end{lemma}
The \textit{signless Laplacian spread} or \textit{$Q$-spread} of $G$ is defined in \cite{Liu} by 
\begin{equation}
S_{Q}(G)=S(Q(G))=q_{1}-q_{n},  \label{qspr}
\end{equation}
where $q_{1}\geq q_{2}\geq \ldots \geq q_{n}$ are the signless Laplacian eigenvalues of $G.$

Some results on the $Q$-spread of a graph can be found for instance in Liu and Liu \cite{Liu} and
Oliveira \textit{et al.} \cite{Carla}.

As $\mu_n$ is always zero, the \textit{Laplacian spread} of $G$ is defined in a slightly different way, see e.g., \cite{Fan},
\begin{equation}
S_{L}(G)=\mu_{1}-\mu_{n-1},  \label{qspr2}
\end{equation}
where $\mu_{1}\geq \mu_{2}\geq \ldots \geq \mu_{n-1} \geq \mu_{n}=0$ are the Laplacian eigenvalues of $G.$

Let $G$ be an $(n,m)$ graph and $\lambda _{1}(\mathcal{L})\geq \lambda _{2}(\mathcal{L})\geq \ldots \geq\lambda _{m}(\mathcal{L})$ be the eigenvalues
of
$\mathcal{L}= \mathcal{L}(G).$
Then the spread of the line graph of $G$ is defined by
\begin{equation}
S_{\mathcal{L}}(G)=\lambda _{1}(\mathcal{L})-\lambda _{m}(\mathcal{L}).
\label{lspr}
\end{equation}
Attending to the relation between the signless Laplacian eigenvalues and the eigenvalues of $\mathcal{L}$ presented in Lemma \ref{381}, from Lemmas \ref{zagreb}, and \ref{model} the following upper bound is a simple consequence:
\begin{eqnarray}
S_{\mathcal{L}}(G)\leq q_{1}-2+\sqrt{Z_{g}(G)-2m-(q _{1}-2)^{2}}\leq 2\sqrt{\frac{Z_{g}(G)}{2}-m}.
\end{eqnarray}

In the next section it will be proved that $S(G)\leq S_{\mathcal{L}}(G)$ for some $(n,m)$ connected graphs.



\section{Bounds for the spread of the line graph and comparisons}
Some results concerning the spread of line graphs can be derived from the previous section and are presented here.
Moreover, it is proved in this section that for an $(n,m)$ connected  graph with $ m > n \geq 4$, the spread of $G$ is a lower bound for the spread of its line graph, where the equality between both spreads holds if and only if $G$ is a regular bipartite graph.
Additionally, an upper bound for the spread of the line graph of $G$ in terms of its vertex connectivity is presented.
Besides that, considering the family of connected graphs with vertex (edge) connectivity at most $k$, with $k>0$, it is characterized the graph that attains the maximum spread of its line graph.

\noindent  Next Lemma gives relations between $S_{Q}(G)$ and $S_{\mathcal{L}}(G)$. It is a direct consequence of Lemma \ref{381}.

 \begin{lemma}
\label{charact}Let $G$ be an $(n,m)$ graph.

\begin{enumerate}
\item If $m=n,$ then $S_{\mathcal{L}}(G)=q_{1}-q_{n}=S_{Q}(G).\ $

\item \label{due} If $m>n,$ then $S_{\mathcal{L}}(G)=q_{1} \geq S_{Q}(G),\ $
with equality if and only if $G$ has a bipartite component.

\item If $m<n,$ then $S_{\mathcal{L}}(G)= q_{1}-q_{m}\leq q_1=S_{Q}(G),$ with equality, for example, for $G=C_n \cup K_2,$ when $n$ is even.
\end{enumerate}
\end{lemma}

\begin{remark}
If $G$ is a connected graph and $m<n$ then $G$ is a tree. Recall that $G$ is bipartite and its signless Laplacian spectrum coincide with its Laplacian spectrum, therefore $S_{\mathcal{L}}(G)=S_{L}(G)$ and $S(G)=2\lambda_{1}(G).$ In some cases, for instance, $G=K_{1,3}$, $S_{\mathcal{L}}(G)=S_{L}(G)=4-1=3<2\sqrt{3}=S(K_{1,3})$.
\end{remark}
The following remark relates $S(G)$ and $S_Q(G)$, and appears in \cite{stanic}.

\begin{remark}
If $G$ is a regular graph then $S\left( G\right) =S_{Q}(G). $ In the general case, these invariants are
incomparable. The computational results based on graphs with a small number
of vertices show that we often have $S\left( G\right) <S_{Q}(G) .$
\end{remark}
In this work we will prove that $S\left( G\right) \leq S_{Q}(G) $ in some cases.


\begin{proposition}
\label{comparison1}
Let $m>n \geq 4.$ Let $G$ be an $(n,m)$ connected graph then $S(G)\leq S_{\mathcal{L}}(G).$ The equality $S(G)=S_{\mathcal{L}}(G)$ holds, if and only if  $G$ is a regular bipartite graph.
\end{proposition}
\begin{proof}
In \cite{ADLR} it was established that $S(G)\leq q_{1}$ with equality if and only if $G$ is bipartite and regular. Thus, for $m>n$, $S_{\mathcal{L}}(G)=q_1\geq S(G)$ with equality if and only if $G$ is a regular bipartite graph.
\end{proof}

\begin{remark}
If $G$ is a unicyclic connected graph with even cycle. From Lemma \ref{charact}, $ S_{\mathcal{L}}(G)=S_{Q}(G)=q_1 .$ In \cite{ADLR} it was established that
$S(G)\leq q_1 $. Thus  $S(G)\leq S_{\mathcal{L}}(G)=S_{Q}(G).$
\end{remark}

When $G$ is a unicyclic graph with an odd cycle the partial result at Theorem \ref{parcial} is proved.
Before proceeding to the proof of Theorem \ref{parcial} we need some previous results.
Let $T$ be a tree. The second smallest Laplacian eigenvalue of $T$ is referred as the \textit{algebraic connectivity of $T$,} and denoted by $a(T)$.
Grone \textit{et al.}, in \cite{lapl2} proved that
\begin{equation*}
a\left( T\right) \leq 0.49
\end{equation*}%
holds for any tree $T$ with at least six vertices. Moreover, the same authors
obtained an upper bound for the algebraic connectivity of a tree in function of  $diam(T),$
\begin{equation}
a\left( T\right) \leq 1-\cos \frac{\pi }{diam(T)+1}\text{.}
\label{upperboundconnec}
\end{equation}

\noindent On the other hand, the following result is in \cite{Domingos}.

\noindent Recall that the \textit{girth} of a graph $G$ is the length of a shortest cycle in $G$.

\begin{theorem}
\cite{Domingos} Let $e$ be an edge of the graph $G$. Let $q_{1},\ldots ,q_{n}$
and $s_{1},\ldots ,s_{n}$ be the signless Laplacian eigenvalues of $G$ and
of $G-e\,,$ respectively. Then
\begin{equation}
0\leq s_{n}\leq q_{n}\leq s_{n-1}\leq \ldots \leq s_{2}\leq q_{2}\leq
s_{1}\leq q_{1}.  \label{interlacing}
\end{equation}
\end{theorem}
\begin{lemma} \cite{Chen,Cvepp}
Let $G$ be a graph with $n$ vertices. Then
\begin{equation} \label{enide2}
2\lambda_1 \leq q_1,
\end{equation}
where $\lambda_1$ is the spectral radius of $G$. The equality holds if and only if $G$ is regular.

\end{lemma}
\begin{theorem}
\label{parcial}
Let $G$ be a connected unicyclic graph with odd girth $g$ and whose maximum
diameter among the induced trees of $G$ is $h$. Let $\lambda _{1}\left( G\right) $ and  $\lambda _{n}\left( G\right) $ be the largest and
smallest eigenvalue of $G,$ respectively.\ If
\begin{equation} \label{enide3}
\lambda _{n}\left( G\right) \geq 1-\cos \frac{\pi }{D_{0}+1}-\lambda
_{1}\left( G\right) ,
\end{equation}%
where
\begin{equation*}
D_{0}=\frac{g+1}{2}+h,
\end{equation*}%
then%
\begin{equation}
S( G) \leq  S_{\mathcal{L}}(G).  \label{ineq5}
\end{equation}
\end{theorem}

\begin{proof}
Note that the girth of $G$ corresponds, in this case, to the number of vertices of the induced cycle of $G$. If $G$ is a cycle then $G$ is regular. Then,$\ S\left( G\right)
=S_{Q}\left( G\right). $ From Lemma \ref{charact}(1) the relation in (\ref{ineq5}) holds in the
equality. If $G$ is not a cycle 
let us consider $T$ as an induced tree of $G$ such that
$diam\left( T\right)=h $. Consider $v$ as the vertex within
the cycle of $G$ which is the root vertex of $T$. Let $w$ be a vertex within the cycle of $G$ placed at the end of the path of length  $\frac{g+1}{2}$ starting in $v$.  Let $u$ be the neighbor of $w$ placed at the end of the path of length  $1+\frac{g+1}{2}$ starting in $v$ within the cycle of $G$. Let $e$ be the edge $wu.\ $Then the deletion of $e$ yields a tree $T_{0}$
with root vertex $w$ and diameter $D_{0}=\frac{g+1}{2}+h$ (see Figure 1). Let $q_{1},\ldots ,q_{n}$
and $s_{1},\ldots ,s_{n}$ be the signless Laplacian
eigenvalues of $G$ and of $G-e$, respectively.
By (\ref{interlacing}) one can see that
\begin{equation}\label{enide1}
q_{n}\leq s_{n-1}=a\left( T_{0}\right) \leq 1-\cos \frac{\pi }{D_{0}+1}.
\end{equation}

From (\ref{enide2}), (\ref{enide3}) and (\ref{enide1}) we have
\begin{equation*}
q_{1}-q_{n}\geq q_{1}-1+\cos \frac{\pi }{D_{0}+1}\geq 2\lambda _{1}-1+\cos
\frac{\pi }{D_{0}+1}\geq \lambda _{1}-\lambda _{n}.
\end{equation*}

But this means that $$ S(G) \leq S_{Q}(G).$$

Hence, by Lemma \ref{charact}(1) again, the relation in (\ref{ineq5}) holds.

\end{proof}

\begin{corollary}
If $G$ is a graph under the conditions of Theorem \ref{parcial}, then $S_{\mathcal{L}}(G)=q_{1}-q_{n}=S_{Q}(G) \geq S(G).\ $
\end{corollary}

The next example illustrates the previous theorem.
\begin{example}
For the graph $G$ depicted in the Figure 1, we have
\begin{align*}
 \lambda_1(G)&=2.17,\ \lambda_n(G)=-2, D_0=7,\
  \cos{\frac{\pi}{D_0+1}}=0.9239.
\end{align*}
Then, the condition in (\ref{enide3}) becomes
$$-2\geq 1- 0.9239-2.17=-2.0939.$$ Moreover, $S_{\mathcal{L}}(G)=S_Q(G)=4.47 \ \text{and}\ S(G)=4.17,$ which check our result.

\begin{figure}[t]
\begin{eqnarray*}
\begin{tikzpicture}\node at (1.5,-0.2) {\textbf{v}};\node at (0.5,-0.2) {\textbf{u}};\node at (-0.3,0.8) {\textbf{w}};
    \tikzstyle{every node}=[draw,circle,fill=black,minimum size=3.5pt,inner sep=0.2pt]
    \draw
                  (0.5,0) node (1) [label=below:] {}
                  (0,0.8) node (2) [label=below:] {}
                  (1,1.5) node (3) [label=below:] {}
                  (2,0.8) node (4) [label=below:] {}
                  (1.5,0) node (5) [label=below:] {}
                  (2.5,0) node (6) [label=below:] {}
                  (3.5,0) node (7) [label=above:] {}
                  (4.5,0) node (8) [label=below:] {}
                  (5.5,0) node (9) [label=below:] {};
        \draw (1)--(2);\draw (2)--(3);\draw (3)--(4);\draw (4)--(5);\draw (5)--(1);\draw (5)--(6);\draw (6)--(7);\draw (7)--(8);\draw (8)--(9);
            \end{tikzpicture}
\end{eqnarray*}
\caption{$g=5,\ h=4, diam\left(T_{0}\right)=7=D_{0}.$}%
\end{figure}

\end{example}

In what follows it is characterized the graph with maximum spread of its line graph, into the family of connected graphs with vertex (edge) connectivity at most $k$, where $k$ is a given positive integer. Some preliminary results are previously presented.

In \cite[Theorem 5]{m23}, the spectrum of the adjacency matrix of the $H\text{-}join$ of regular graphs is obtained. Consider the graph $K_{i} \vee( K_{k} \cup K_{n-k-i}).$ As this graph can be seen as the $P_3$-join of the family of graphs $\{K_{i}, K_{k}, K_{n-k-i}\}$, from  \cite[Theorem 5]{m23}, Corollary \ref{Corolario2} below is immediate.

\begin{corollary} \label{Corolario2}
The
eigenvalues of the line graph of $K_{i} \vee( K_{k} \cup K_{n-k-i})$ 
are
\begin{eqnarray*}
n+\dfrac{k}{2}-4+\dfrac{1}{2}\sqrt{(2n-k)^2+16i(k-n+i)}, \\
n+\dfrac{k}{2}-4-\dfrac{1}{2}\sqrt{(2n-k)^2+16i(k-n+i)},
\end{eqnarray*}
both with multiplicity $1$ and
\begin{equation*}
n-4, k+i-4,n-i-4,-2,
\end{equation*}
with multiplicities
$k, i-1, n-k-i-1$ and $m-n$, respectively.

\end{corollary}

\noindent The following result is a direct consequence of Corollary \ref{Corolario2} and  (\ref{lspr}).

\begin{proposition}\label{101}
\begin{equation*}
S_{\mathcal{L}}\left(K_{i} \vee( K_{k} \cup K_{n-k-i})\right) = n-2+\frac{1%
}{2}k+\sqrt{\left( 2n-k\right) ^{2}+16i\left( k-n+i\right) }.
\end{equation*}
\end{proposition}

It is known that the spectral radius of a nonnegative irreducible matrix
increases if any of its entries increases. From this fact, we have the following result.

\begin{lemma} \label{Lemma5}
Let $G$ be a connected graph then $$q_1(G)<q_1(G+e),$$
where $G+e$ denotes the graph that results from $G$ adding an edge $e.$
\end{lemma}

From Lemma \ref{Lemma5} above the following result is immediate.

\begin{lemma}
\label{12} Let $G$ be a connected graph. Then
\begin{equation*}
\lambda_1(\mathcal{L}(G))<\lambda_1(\mathcal{L}(G+e)).
\end{equation*}
\end{lemma}

The following result is a direct consequence of Cauchy's Interlacing Theorem \cite{m25}, and Lemma \ref{12}.
\begin{theorem}
\label{6} Let $G$ be a connected graph. Then
\begin{equation*}
S_{\mathcal{L}}(G)< S_{\mathcal{L}}(G+e).
\end{equation*}
\end{theorem}
The \textit{vertex connectivity} (or just \textit{connectivity}) of a graph $G$, denoted by $\kappa \left( G\right) $, is the minimum number of vertices
of $G$ whose deletion disconnects $G$. Let $\mathcal{F}_{n}$ be the family of connected graphs on $n$ vertices. Let%
\begin{equation*}
\mathcal{V}_{n}^{k}=\left\{ G\in \mathcal{F}_{n}:\kappa \left( G\right) \leq k\right\}
.
\end{equation*}

The following result characterizes the graph with maximum spread of its line graph into the family of connected graphs with vertex connectivity at most $k$, where $k$ is a given positive integer.
\begin{theorem}\label{100}
Let $G\in \mathcal{V}_{n}^{k}$. Then,
\begin{equation}
S_{\mathcal{L}}(G)\leq n-2+\frac{1}{2}k+\sqrt{(2n-k)^{2}+16(k-n+1)}
\label{8}
\end{equation}%
with equality if and only if $G\cong K_{1} \vee (K_{k} \cup K_{n-k-1}).$
\end{theorem}

\begin{proof}
Let $G\in \mathcal{V}_{n}^{k}$ be such that $\mathcal{L}(G)$ has the largest
spread among all the graphs $\mathcal{L}(H)$ with $H\in \mathcal{V}_{n}^{k}.$ Let $U\subset V(G), $ such that $|U|\leq k$ and  $G-U$ is a disconnected graph. Let $
X_{1},X_{2},\ldots ,X_{l},$ be the connected components of $G-U$. We claim
that $l=2$. If $l>2$ then we can construct a graph $H=G+e$ where $e$ is an
edge connecting a vertex in $X_{1}$ with a vertex in $X_{2}$. Clearly, $H\in
\mathcal{V}_{n}^{k}$. By Theorem \ref{6},
$S_{\mathcal{L}}(G)< S_{\mathcal{L}}(H),$
which is a contradiction.
Therefore $l=2$, that is, $G-U=X_{1}\cup X_{2}$. Recall that $|U|\leq k$. Now, we claim that $|U|=k$.
Suppose $|U|<k$, then we construct a graph $H=G+e$ where $e$ is an edge joining a vertex $u\in V(X_{1})$ with a vertex $v\in V(X_{2})$. We see that $H-U$ is a
connected graph and the deletion of the vertex $u$ disconnects $H-U$. This
tells us that $H\in \mathcal{V}_{n}^{k}$. By Theorem \ref{6}, $S_{\mathcal{L}%
}(G)< S_{\mathcal{L}}(H)$, which is a contradiction. Then, $|U|=k$.
Therefore, $G-U=X_{1}\cup X_{2}$ and $|U|=k$. Let $|X_{1}|=i$ then $|X_{2}|=n-k-i$.

We claim that repeated applications of Theorem \ref{6} enable us to write
\begin{equation*}
G(i)\cong K_{i}\vee (K_{k}\cup K_{n-k-i}) \cong G
\end{equation*}
for some $1\leq i\leq \lfloor \frac{n-k}{2}\rfloor .$ In fact, this means that if $Y_3$ is the induced subgraph of $G$ obtained from the vertices in $U$ then, there would be an edge
$$e \in [E(\overline{Y_{1}} \vee \overline{Y_{3}}) \cup E(\overline{Y_{2}} \vee \overline{Y_{3}})] - E(G).$$
Therefore, it is possible to construct a new graph $H=G+e$. Clearly, $H \in \mathcal{V}_n^k$. By Theorem 12, $S_{\mathcal{L}}(G)<S_{\mathcal{L}}(H)$ which is a contradiction.

Until this point, we
have proved $S_{\mathcal{L}}(G)\leq S_{\mathcal{L}}(G(i)),$ for all $G\in
\mathcal{V}_{n}^{k}$. We now search for a value of $i$ for which $S_{\mathcal{L%
}}(G(i))$ is maximum.

From Theorem \ref{101}, it is obtained
\begin{equation*}
S_{\mathcal{L}}(G(i))=n-2+\frac{1}{2}k+\sqrt{(2n-k)^{2}+16i(k-n+i)}.
\end{equation*}
Define the function,
\begin{equation*}
g(x)=n-2+\frac{1}{2}k+\sqrt{(2n-k)^{2}+16x(k-n+x)}
\end{equation*}
where $1\leq x\leq \lfloor \frac{n-k}{2}\rfloor$. In this interval $g$ is a
strictly decreasing function. Consequently
$S_{\mathcal{L}}(G)\leq S_{\mathcal{L}}(G(1))$
for all $G\in \mathcal{V}_{n}^{k}$.
Moreover, since $G(1)\cong K_{k} \vee (K_{1} \cup K_{n-k-1})$ and
$S_{\mathcal{L}}(G(1))=n-2+\frac{1}{2}k+\sqrt{(2n-k)^{2}+16(k-n+1)}$ the
equality in (\ref{8}) holds if and only if $G\cong K_{1} \vee (K_{k} \cup K_{n-k-1}).$
\hfill
\end{proof}
\vspace{0.5 cm}
We recall now the definition of \textit{edge-connectivity} of $G$, denoted here by $\varepsilon(G)$, as the minimum number of edges whose deletion disconnects $G$. Note also that in graphs that represent communication or transportation networks, the edge-connectivity is an important measure of reliability.
\newpage
Let

\begin{equation*}
\mathcal{E}_n^k=\{G\in \mathcal{F}_n : \varepsilon(G)\leq k\}.
\end{equation*}

It is well known that $\kappa(G)\leq \varepsilon(G)\leq \delta(G),$ where $\delta(G)$ denotes the minimum degree of $G$, see \cite{Harary, Whitney}.

\begin{corollary}
\label{incl1}
Let $G\in\mathcal{E}_n^k$. Then,
\begin{equation*}
S_{\mathcal{L}}(G)\leq n-2+\frac{1}{2}k+\sqrt{(2n-k)^{2}+16(k-n+1)}
\end{equation*}
with equality if and only if $G\cong K_{1} \vee (K_{k} \cup K_{n-k-1}).$
\end{corollary}

\begin{proof}
Since $\kappa(G)\leq \varepsilon(G)$, it follows that $\mathcal{E}_n^k\subset \mathcal{V}_n^k$. Let $G\in \mathcal{E}_n^k$. Then $G\in \mathcal{V}_n^k$.
From this, and the fact that $G\cong K_{1} \vee (K_{k} \cup K_{n-k-1})\in \mathcal{E}_n^k$, the result follows.
\hfill
\end{proof}

Let $\Delta_n^k=\{G\in\mathcal{F}_n:\delta(G)\leq k\}$. Then,
$\Delta_n^k\subseteq\mathcal{V}_n^k$. Moreover, the graph $ K_{1} \vee (K_{k} \cup K_{n-k-1})$ have minimum degree $k$. Then, attending to Theorem \ref{100},
we can also obtain the following result.

\begin{corollary}
\label{minimum}
Let $G\in\Delta_n^k$. Then,
\begin{equation*}
S_{\mathcal{L}}(G)\leq n-2+\frac{1}{2}k+\sqrt{(2n-k)^{2}+16(k-n+1)}
\end{equation*}
with equality if and only if $G\cong K_{1} \vee (K_{k} \cup K_{n-k-1})$.
\end{corollary}

\section{Lower bounds for different spreads of total graphs}
In this section we present some lower bounds for different spreads of total graphs. The tools used here were the interlacing eigenvalues \cite[p. 594]{haemers2} and the definition of equitable partitions that we recall below.

For $1\leq i,j\leq k$, let us consider the $n_{i}\times n_{j}$ matrices $%
M_{ij}.$ Let $n=\sum\limits_{i=1}^{k}n_{i}$\ and suppose that $%
M_{ij}=M_{ji}^{t}$,\ for all $\left( i,j\right) $.\ We consider the
partitioning of the $n\times n$ symmetric matrix into blocks
\begin{equation}
M=
\begin{pmatrix}
M_{ij}
\end{pmatrix}_{1\leq i,j\leq k}
\text{.}  \label{blocks2}
\end{equation}%
Let us denote by $\mathbb{J}_{pq}$ the all ones matrix of order $p\times q$ and
simply by $\mathbb{J}_{p}$ the all ones vector of order $p\times 1$.\ The
quotient matrix $\overline{M}=\left( m_{ij}\right) \ $of $M$ is the $k\times
k$ matrix whose $\left( i,j\right) $-entry is the average of the row sums of
$M_{ij}$. More precisely%
\begin{equation}
m_{ij}=\frac{1}{n_{i}}\left( \mathbb{J}_{n_{i}}^{t}M_{ij}\mathbb{J}%
_{n_{j}}\right) ,\qquad \text{\ for }1\leq i,j\leq k.  \label{bij}
\end{equation}%
The partitioning into blocks of $M$ is called regular (or equitable) if each
block $M_{ij}$ of $M$ has constant row sum. Note that in this case $%
\overline{M}$ corresponds to the row sums matrix. According to
\cite{haemers2}, if $M$ is regular, then all the eigenvalues
of $\overline{M}$ are eigenvalues of $M$.

\begin{theorem}
\label{Haemers copy(1)}\cite{haemers2} Suppose that $\overline{M}$ is the quotient
matrix of a partitioned symmetric matrix $M$, then the eigenvalues of $%
\overline{M}$ interlace the eigenvalues of $M$.\ Moreover, if the
interlacing is tight, then the partition of $M$ is regular. On the other
hand, if the matrix $M$ is regularly partitioned, then the eigenvalues of $%
\overline{M}$ are eigenvalues of $M$.\
\end{theorem}

Let $G$ be an $(n,m)$ graph. We recall that the incidence matrix of a graph $H$ is a matrix $R$ whose rows and columns are indexed by the vertices and edges of $H$, respectively.
The $(i,j)$- entry of $R$ is $r_{ij}=0$ if $i$ is not incident with $j$ and $r_{ij}= 1$ if $j$ is incident with $i.$

The adjacency matrix of
$
\mathcal{T}(G)$ is given by%
\begin{equation*}
A\left( \mathcal{T}\left( G\right) \right) =\left[
\begin{array}{cc}
A\left( G\right) & R \\
R^{t} & A\left( \mathcal{L}\left( G\right) \right)%
\end{array}%
\right]
\end{equation*}%
where $R$ is the incidence matrix of $G$.




 Directly from the definition, we conclude that for $u\in V\left( \mathcal{T}\left( G\right) \right) $

\begin{equation*}
d_{\mathcal{T}\left( G\right) }\left( u\right) =\left\{
\begin{array}{lcl}
2d_{G}\left( u\right) & \text{if } & u\in V\left( G\right) \\
d_{\mathcal{L}(G)}\left( u\right) +2 & \text{if } & u\in V\left( \mathcal{L}%
(G)\right)%
\end{array}%
\right. \text{.}
\end{equation*}

Hence, the diagonal matrix of vertex degrees of $\mathcal{T}\left( G\right) $ is%
\begin{equation*}
D\left( \mathcal{T}\left( G\right) \right) =\left[
\begin{array}{cc}
2D\left( G\right) & 0 \\
0 & D\left( \mathcal{L}(G)\right) +2I_{m}%
\end{array}%
\right] \text{.}
\end{equation*}

\begin{remark} \label{minmax} Let $\delta(G)$ and $\Delta(G)$ the minimum and maximum vertex degree of a graph $G$, respectively. It is immediate that
\begin{eqnarray}
\delta(\mathcal{T}(G))&=&2\delta(G)\\
\Delta(\mathcal{T}(G))&=& 2\Delta(G).
\end{eqnarray}
\end{remark}

Now a lower bound for  the signless Laplacian spread of total graphs is presented.

\begin{theorem}
Let $G$ be an $(n,m)$ connected graph. Then,
\begin{equation*}
S_{Q}\left( \mathcal{T}\left( G\right) \right) \geq 2\sqrt{\left( \frac{3m}{n%
}-\frac{Z_{g}(G)}{m}\right) ^{2}+\frac{10m}{n}-\frac{2Z_{g}(G)}{m}+1}\text{.}
\end{equation*}
\end{theorem}

\begin{proof}
Applying Theorem \ref{Haemers copy(1)}
to the signless Laplacian matrix, $Q\left( \mathcal{T}\left( G\right) \right)$
\begin{eqnarray}
Q\left( \mathcal{T}\left( G\right) \right) &=&D\left( \mathcal{T}\left(
G\right) \right) +A\left( \mathcal{T}\left( G\right) \right)  \notag \\
&=&\left[
\begin{array}{cc}
2D\left( G\right) &  0\\
 0 & D\left( \mathcal{L}(G)\right) +2I_{m}%
\end{array}%
\right] +\left[
\begin{array}{cc}
A\left( G\right) & R \\
R^{t} & A\left( \mathcal{L}\left( G\right) \right)%
\end{array}
\right]  \notag \\
&=&\left[
\begin{array}{cc}
Q\left( G\right) & 0 \\
0 & Q\left( \mathcal{L}(G)\right)%
\end{array}%
\right] +\left[
\begin{array}{cc}
D\left( G\right) & R \\
R^{t} & 2I_{m}%
\end{array}
\right] \text{.}  \label{13}
\end{eqnarray}%

Thus, the quotient matrix of $Q\left( \mathcal{T}\left( G\right) \right)$ becomes
\begin{equation*}
\overline{M}_{Q}:=\overline{M}\left( Q\left( \mathcal{T}\left( G\right) \right) \right) =\left[
\begin{array}{cc}
\frac{6m}{n} & \frac{2m}{n} \\
2 & 2+\frac{4\theta}{m}%
\end{array}%
\right] \text{,}
\end{equation*}
where $\theta$ stands for the number of edges of the line graph. The characteristic equation of $\overline{M}_{Q}$ is
\begin{equation*}
\lambda^{2}-2\lambda\left( 1+\frac{3m}{n}+\frac{2\theta}{m}\right) +\frac{8}{n}\left(m+3\theta\right)=0
\text{.}
\end{equation*}

Solving this equation, we have
\begin{eqnarray*}
\lambda_{1}\left(\overline{M}_{Q}\right)&=&\left( \frac{3m}{n}+\frac{2\theta}{m}+1\right) + \sqrt{\left(\frac{3m}{n}\right)^{2}+\left( \frac{2\theta}{m}\right)^{2}+1-\frac{12\theta}{n}-
\frac{2m}{n}+\frac{4\theta}{m}} \text{,} \\
\lambda_{2}\left(
\overline{M}_{Q}\right)&=&\left( \frac{3m}{n}+\frac{2\theta}{m}+1\right) - \sqrt{\left(\frac{3m}{n}\right) ^{2}+\left( \frac{2\theta }{m}\right) ^{2}+1-\frac{12\theta }{n}-%
\frac{2m}{n}+\frac{4\theta }{m}}\text{.}
\end{eqnarray*}

Therefore,
\begin{equation*}
S_{Q}\left( \mathcal{T}\left( G\right) \right) \geq 2\sqrt{\left(\frac{3m%
}{n}\right) ^{2}+\left( \frac{2\theta }{m}\right) ^{2}+1-\frac{12\theta }{n}-%
\frac{2m}{n}+\frac{4\theta }{m}.}
\end{equation*}

Recalling from (\ref{nÂ°edges}) the number of edges of the line graph of a graph $G$ with $n$ vertices and $m$ edges, the result follows.
\end{proof}

Moreover, the following lower bound for the spread of total graphs is obtained.

\begin{theorem}
Let $G$ be a connected graph on $n$ vertices and $m$ edges. Then
\begin{equation*}
S\left( \mathcal{T}\left( G\right) \right) \geq \sqrt{\left( \frac{%
2m^{2}+n\left( Z_g\left( G\right) -2m\right) }{mn}\right) ^{2}-\frac{8\left(
Z_g\left( G\right) -4m\right) }{n}}\text{.}
\end{equation*}
\end{theorem}

\begin{proof} The quotient matrix of $A\left( \mathcal{T}\left(
G\right) \right) $ is given by%
\begin{equation*}
\overline{M}_{A}:=\overline{M}\left( A\left( \mathcal{T}\left( G\right) \right) \right) =\left[
\begin{array}{cc}
\frac{2m}{n} & \frac{2m}{n} \\
2 & \frac{Z_{g}\left( G\right) -2m}{m}
\end{array}
\right].
\end{equation*}

Then, the characteristic equation of $\overline{M}_{A}$ is%
\begin{equation*}
\left( \lambda -\frac{2m}{n}\right) \left( \lambda -\frac{Z_{g}\left(
G\right) -2m}{m}\right) -\frac{4m}{n}=0\text{.}
\end{equation*}%
Solving the equation, we have%
\begin{eqnarray*}
\lambda _{\pm }\left( \overline{M}_{A}\right) &=& \frac{1}{2} \psi \pm \sqrt{\psi ^{2}-\frac{8( Z_{g}(G) -4m) }{n}},
\end{eqnarray*}
where $\psi: =\frac{2m^{2}+n \left ((Z_{g}\left( G\right) -2m\right)}{mn}.$

Therefore,%
\begin{equation*}
S\left( \mathcal{T}\left( G\right) \right) \geq \sqrt{\left( \frac{%
2m^{2}+n\left( Z_{g}\left( G\right) -2m\right) }{mn}\right) ^{2}-\frac{%
8\left( Z_{g}\left( G\right) -4m\right) }{n}}\text{.}
\end{equation*}
\end{proof}
The Laplacian matrix of $\mathcal{T}\left( G\right) $ is
\begin{eqnarray}
L\left( \mathcal{T}\left( G\right) \right) &=&D\left( \mathcal{T}\left(
G\right) \right) -A\left( \mathcal{T}\left( G\right) \right)  \notag \\
&=&\left[
\begin{array}{cc}
2D\left( G\right) &  \\
& D\left( \mathcal{L}(G)\right) +2I_{m}%
\end{array}%
\right] -\left[
\begin{array}{cc}
A\left( G\right) & R \\
R^{t} & A\left( \mathcal{L}\left( G\right) \right)%
\end{array}%
\right]  \notag \\
&=&\left[
\begin{array}{cc}
L\left( G\right) & 0 \\
0 & L\left( \mathcal{L}(G)\right)%
\end{array}%
\right] +\left[
\begin{array}{cc}
D\left( G\right) & -R \\
-R^{t} & 2I_{m}%
\end{array}%
\right] \text{.}  \label{3}
\end{eqnarray}%

\begin{theorem}
\label{lwbsprd}
Let $G$ be a connected graph on $n$ vertices, $m$ edges and smallest degree $\delta. $
Then,

\begin{equation*}
S_{L}\left( \mathcal{T}\left( G\right) \right) \geq \left|\frac{2m+2n}{n}-2\delta\right|.
\end{equation*}
\end{theorem}

\begin{proof}
Applying Theorem \ref{Haemers copy(1)} to the Laplacian matrix $L\left( \mathcal{T}\left( G\right) \right)$ partitioned as in (\ref{3}), the quotient matrix becomes
\begin{equation*}
\overline{M}_{L}:=\overline{M}\left( L\left( \mathcal{T}\left( G\right) \right) \right) =\left[
\begin{array}{cc}
\frac{2m}{n} & -\frac{2m}{n} \\
-2 & 2%
\end{array}
\right].
\end{equation*}
The characteristic equation of $\overline{M}_{L}$ is
\begin{equation*}
\left( \lambda -\frac{2m}{n}\right) \left( \lambda -2\right) -\frac{4m}{n}=0.
\end{equation*}
Solving this equation, we have
\begin{equation*}
\lambda _{1}\left(\overline{M}_{L}\right) =\frac{2m+2n}{n}\ \text{and }\lambda _{2}\left(
\overline{M}_{L} \right) =0.
\end{equation*}
By interlacing of the eigenvalues,  \cite[p. 154]{haemers2},
\begin{equation*}
\mu_1\left( \mathcal{T}\left( G\right) \right) \geq \frac{2m+2n}{n}\geq \mu_{n-1}\left( \mathcal{T}\left( G\right) \right).
\end{equation*}

It is known that (see \cite{Fiedler}), if $H$ is a non-complete graph, $ \mu_{n-1} (H)\leq \kappa(H)$. Moreover
$\kappa (H) \le \delta(H)$, then
\begin{eqnarray}
\mu_{n-1}(H) \le \delta (H).
\end{eqnarray} Therefore, using the previous inequality with $\mathcal{T}(G)$ instead of $H$ and recalling Remark \ref{minmax} the result follows.
\end{proof}

\section{Bounds for the spread of the total graph of a regular graph}

In this section we present a lower and upper bound for the spread of the total graph of a regular graph. These bounds are obtained in function of the spread of the graph.

For an $r$-regular graph $G$, in 1973, Cvetkovi\'{c} \cite{cvetc}, obtained the following result.

\begin{theorem}  \rm{\cite{cvetc}}
\label{lem1} Let $G$ be a regular graph of order $n$ and degree $r$ with eigenvalues $$\lambda_{1}\geq \cdots \geq \lambda_{n}.$$ Then
the eigenvalues of $\mathcal{T}\left( G\right) $ are
$$\dfrac{2\lambda _{i}+r-2\pm \sqrt{4\lambda _{i}+r^{2}+4}}{2}, i=1,\ldots,n, $$
with multiplicity one
and $-2$ with multiplicity $\dfrac{n(r-2)}{2}.$
\end{theorem}




Let $n\geq2$. It is well known that if $\lambda_n(G)$ is the smallest eigenvalue of a connected graph $G$ then
\begin{equation}
\label{acotamiento}
\lambda_n(G)\leq -1.
\end{equation}

The following result identifies the smallest eigenvalue of the total graph of  a regular graph with $n$ vertices and vertex degree $r$.

\begin{lemma}
Let $G$ be a connected regular graph of order $n$ and degree $r$, $r\geq 3$.
Then
\begin{equation*}
\lambda _{\frac{n(r+2)}{2}}(\mathcal{T}(G))=\dfrac{2\lambda _{n}+r-2-\sqrt{%
4\lambda _{n}+r^{2}+4}}{2}  \label{5}
\end{equation*}
where $\lambda_n$ is the smallest eigenvalue of $G$.
\end{lemma}

\begin{proof}
Let $G$ be a connected regular graph of order $n$ and degree $r$, $r\geq 3$. Then (\ref{acotamiento}) holds.
By Perron-Frobenius's Theory,
\begin{equation*}
-r\leq \lambda _{n}.
\end{equation*}%
Thus,
\begin{equation*}
\lambda _{n}^{2}+\lambda _{n}(r+1)+r\leq 0.
\end{equation*}%
Now, suppose that
\begin{equation*}
2\lambda _{n}+r+2\geq 0.
\end{equation*}%
Also,
\begin{equation*}
2\lambda _{n}+r+2\leq \sqrt{4\lambda _{n}+r^{2}+4}.
\end{equation*}%
Therefore,
\begin{equation}
\dfrac{2\lambda _{n}+r-2-\sqrt{4\lambda _{n}+r^{2}+4}}{2}\leq -2.
\label{dsgldd}
\end{equation}%
On the other hand, if
\begin{equation*}
2\lambda _{n}+r+2<0,
\end{equation*}%
then
\begin{equation*}
\dfrac{2\lambda _{n}+r-2-\sqrt{4\lambda _{n}+r^{2}+4}}{2}<-2.
\end{equation*}%
Since, the functions
\begin{equation*}
f_{\pm }(x)=\dfrac{2x+r-2\pm \sqrt{4x+r^{2}+4}}{2}
\end{equation*}%
are strictly increasing in the interval $(-r,r)$, the result follows.
\end{proof}

\medskip


Now, if $G$ is a regular graph then $\lambda _{1}\left( G\right) =r$ and
$\lambda _{1}\left( \mathcal{T}\left( G \right) \right) =2r$,
and the next corollary can be obtained.

\begin{corollary}Let $G$ be a connected regular graph of order $n$ and degree $r\geq 2$. Then
\begin{eqnarray*}
S\left( \mathcal{T}\left( G\right) \right) &=&\dfrac{2S\left( G\right) +r+2+\sqrt{4\lambda _{n}+r^{2}+4}}{2} \\
&\geq &\dfrac{2S\left( G\right) +\lambda _{n}+2+\sqrt{4\lambda
_{n}+r^{2}+4}}{2}\text{.}
\end{eqnarray*}%
\end{corollary}
\begin{proof}
\begin{eqnarray*}
S\left( \mathcal{T}\left( G\right) \right) &=&2r-\left( \dfrac{2\lambda
_{n}+r-2-\sqrt{4\lambda _{n}+r^{2}+4}}{2}\right)\text{.}
\end{eqnarray*}%
\end{proof}
\noindent Since,
\begin{equation*}
S\left( \mathcal{T}\left( G\right) \right) =\dfrac{2S\left( G\right)
+r+2+\sqrt{4\lambda _{n}+r^{2}+4}}{2}\text{.}
\end{equation*}
By (\ref{dsgldd}),
\begin{equation*}
S\left( G\right) +\sqrt{4\lambda _{n}+r^{2}+4}-\lambda _{n}\geq
S\left( \mathcal{T}\left( G\right) \right) \text{.}
\end{equation*}

So, we have proven the next result.

\begin{theorem}
Let $G$ be a regular graph of order $n$ and degree $r$. Then
\begin{equation*}
\dfrac{2S\left( G\right)
+\lambda _{n}+2+\sqrt{4\lambda _{n}+r^{2}+4}}{2} \leq S\left( \mathcal{T}\left( G\right) \right) \leq S\left( G\right) + \sqrt{4\lambda _{n}+r^{2}+4} -\lambda_{n}.
\end{equation*}
\end{theorem}

\noindent \textbf{Acknowledgments:} Enide Andrade was supported in part by the Portuguese Foundation for Science and Technology (FCT-Funda\c{c}\~{a}o para a Ci\^{e}ncia e a Tecnologia), through CIDMA - Center for Research and
Development in Mathematics and Applications, within project UID/MAT/04106/2013.
Exequiel Mallea-Zepeda was supported by Proyecto UTA-Mayor, 4740-18, Universidad de Tarapac\'a, Chile. E. Lenes, E. Mallea-Zepeda and J. Rodr\'iguez also express their gratitude  to professor Ricardo Reyes for his careful reading and suggestions of the final form of this work.
  



\begin{thebibliography}{99}


\bibitem{ADLR} E. Andrade, G. Dahl, L. Leal, M. Robbiano, New bounds for the signless Laplacian spread, arxiv.org/abs/1805.11803.

\bibitem{Belizad} M. Belizad, A characterization of total graphs, Proceeding of the American Mathematical Society, 26 (1970) 383-389.

\bibitem{BelizadRadjavi} M. Belizad, H. Radjavi, Structure of regular total graphs, J. London Math. Soc. 44 (1969) 433-436 MR 38 4344.



\bibitem{HaemerBk} A. Brouwer, W. Haemers, Spectra of graphs, Springer-Verlag, Berlin, 2012.

\bibitem{Domingos} D. M. Cardoso, D. Cvetkovi\'{c}, P. Rowlinson, S. K. Simi\'{c}, A sharp lower bound for the least eigenvalue of the signless
Laplacian of a non-bipartite graph, Lin. Algebra Appl. 429 (2008)
2770\text{-}2780.

\bibitem{m23} D. M. Cardoso, M. Aguieiras de Freitas, E. A. Martins, M.
Robbiano, Spectra of graphs obtained by a generalization of the join graph operation, Discr. Math. 313 (2013) 733-741.
\bibitem{Chen} Y. Chen, Properties of spectra of graphs and line graphs, Appl. Math. J. Ser. B 3 (2002) 371-376.


\bibitem{cvetc} D. M. Cvetkovi\'{c}, Spectrum of the total graph of a graph, Publ. Inst. Math.(Beograd), 16 49-52, 1973.

\bibitem{Cvepp} D. Cvetkovi\'{c}, P. Rowlinson, S. Simi\'{c}, Eigenvalue bounds for the signless Laplacian, Publ. Inst. Math. (beograd)(N.S.) 81 (95) (2017) 11-27.


\bibitem{Fan} Y. Z. Fan, J. Xu, Y. Wang, D. Liang, The Laplacian spread of a tree, Discr. Math.
Theor. Comput. Sci. 10 (2008) 79-86.




\bibitem{Fiedler} M. Fiedler, Algebraic connectivity of graphs, Czechoslovak Math. J. 23 (1973) 298-305.

\bibitem{G} D. A. Gregory, D. Heshkowitz, S. J. Kirkland, The spread of the spectrum of a graph, Linear Algebra Appl., 332-334 (2001) 23-35.

\bibitem{lapl1} R. Grone, R. Merris, The Laplacian spectrum of a graph II,
SIAM J. Discr. Math. 7 (1994) 221-229.

\bibitem{lapl2} R. Grone, R. Merris, V. S. Sunder, The Laplacian spectrum of
a graph, SIAM J. Matrix Anal. Appl. 11 (1990) 218-238.


\bibitem{GutmanDas} I. Gutman, K. C. Das, The first Zagreb index 30 years
after, MATCH Comm. Math. Comput. Chem. 50 (2004) 82-92.


\bibitem{haemers2} W. Haemers, Interlacing eigenvalues and graphs, Linear
Algebra Appl. 227-228 (1995) 593-616.

\bibitem{Harary} F. Harary, Graph theory, Reading, MA Addison-Wesley, p. 43, 1994.


\bibitem{m25} S. Hwang, Cauchy's interlace theorem for eigenvalues of
hermitian matrices, The American Mathematical Monthly, Vol. 111, 
157-159, (2004).

\bibitem{Liu} M. Liu, B. Liu, The signless Laplacian spread, Linear Algebra
and Appl. 432 (2010) 505-514.

\bibitem{Carla} C. S. Oliveira, L. S. de Lima, N. M. M. Abreu, S. Kirkland, Bounds on the Q-spread of a graph,
Linear Algebra Appl. 432, 9, (2010) 2342-2351.

\bibitem{stanic} Z. Stani\'c, Inequalities for graph eigenvalues, London, Mathematical Society. Lecture Note Series 423. Cambridge University Press.

\bibitem{Whitney} H. Whitney, Congruent graphs and the connectivity of graphs,  Amer. J. Math. 54, (1932) 150-168.


\bibitem{tenacity} L. Yinkut, W. Zongtian, Y. Xiaokut, L. Erquiang, Tenacity of total graphs, Int. J. Found. Comput. Sci. 25, (2014) 553.

\end{thebibliography}
\end{document}